\documentclass[reqno, 12pt]{amsart}
\usepackage[letterpaper,hmargin=1in,vmargin=1in]{geometry}
\usepackage{amsmath, amssymb, amsthm, verbatim, url}
\usepackage{graphicx}
\usepackage{enumerate}

\newcommand \R {\mathbb{R}}

\newcommand \Oh {\mathcal{O}}

\DeclareMathOperator \Tr {Tr}

\DeclareMathOperator \spec {spec}

\newtheorem{lem}{Lemma}
\newtheorem*{thm}{Theorem}

\theoremstyle{definition}

\newtheorem{rem}[lem]{Remark}

\numberwithin{equation}{section}
\numberwithin{lem}{section}
\numberwithin{Defn}{section}

\title
{A Fulling-Kuchment theorem for the $1D$ harmonic oscillator}

\author[Victor Guillemin]
{Victor Guillemin}
\address{Department of Mathematics, Massachusetts Institute of Technology, Cambridge, MA
02139-4397, U.S.A.}
\email{vwg@math.mit.edu}
\author[Hamid Hezari]
{Hamid Hezari}
\address{Department of Mathematics, Massachusetts Institute of Technology, Cambridge, MA
02139-4397, U.S.A.}
\email{hezari@math.mit.edu}

\thanks{The first author is supported in part by NSF grant DMS-1005696 and the second author is partially supported by NSF
grant DMS-0969745.}
\date{August 31, 2011}

\begin{document}
\begin{abstract}
We prove that there exists a pair of non-isospectral $1D$ semiclassical Schr\"odinger operators whose spectra agree up to $\Oh(h^\infty)$. In particular, all their semiclassical trace invariants are the same. Our proof is based on an idea of Fulling-Kuchment  and Hadamard's variational formula applied to suitable perturbations of the harmonic oscillator.
\end{abstract}
\maketitle

\section{Introduction}
 This paper concerns the $1D$ semiclassical Schr\"odinger operator  
\[
P_{V} =- h^2\frac{d^2}{dx^2} + V(x), \qquad h>0,
\] where the potential $V$ (which is always assumed to be independent of $h$) satisfies:
\[
V \in C^\infty(\R;\R), \qquad \lim_{|x| \to \infty} V(x) = \infty.
\] 
For any $h>0$, the spectrum of $P_{V}$ on $\R$ is discrete and simple, and we write it as 
\[
\spec(P_{V}) = (\lambda_j)_{j=1}^\infty, \qquad  \lambda_1<\lambda_2<\lambda_3< \cdots \rightarrow \infty.
\]
Each $\lambda_j$ depends on $h$, but we do not include this in the notation. We denote by $u_j$ the corresponding eigenfunctions (which also depend on $h$), so that
\[
P_{V} u_j = \lambda_j u_j, \qquad u_j \in L^2(\mathbb R).
\]
Our main result, which was conjectured by Colin de Verdi\`ere in \cite{c}, is:
\begin{thm}\label{Theorem} There exists a pair of potentials $V^{\pm}(x) \in C^\infty (\mathbb R)$ with $V^{\pm}(x) \geq 0$ such that the operators
\begin{equation} P_{V^\pm}=-h^2 \frac{d^2}{dx^2} + V^{\pm}(x) \end{equation}
satisfy``$\spec(P_{V^+})=\spec(P_{V-})$ up to $\Oh (h^\infty)$" (so in particular they have the same semiclassical trace invariants) and such that the ground state eigenvalues $\lambda_1^+$ and $\lambda^-_1$ are different for all $h>0$ except possibly for a sequence $h_k \to 0$.
\end{thm} Here by ``$\spec(P_{V^+})=\spec(P_{V-})$ up to $\Oh(h^\infty)$'' we mean that for every $E>0$ and $N>0$ there exists a constant $C$ such that
\begin{equation}\label{spectrums}
\sup_{\{\lambda_j^\pm<E\}}|\lambda_j^+-\lambda_j^-| \leq C h^{N}.
\end{equation}
The inverse spectral problem for $1D$ semiclassical Schr\"odinger operators asks whether $\spec(P_{V})$ determines $V$ uniquely up to a translation $x \to x-x_0$ or reflection $x \to -x$. The main tools in studying inverse spectral problems are trace invariants such as heat, wave or Schr\"odinger trace invariants. We recommend the surveys  \cite{zelditchsurvey} by Zelditch, and \cite{dh} by Datchev and the second author for applications of different kinds of trace formulas in inverse spectral results. However the theorem above shows the limitations of semiclassical trace invariants meaning that two Schr\"odinger operators can have the same semiclassical invariants but have different spectra. 

\begin{figure}[htbp]
\includegraphics[width=12cm]{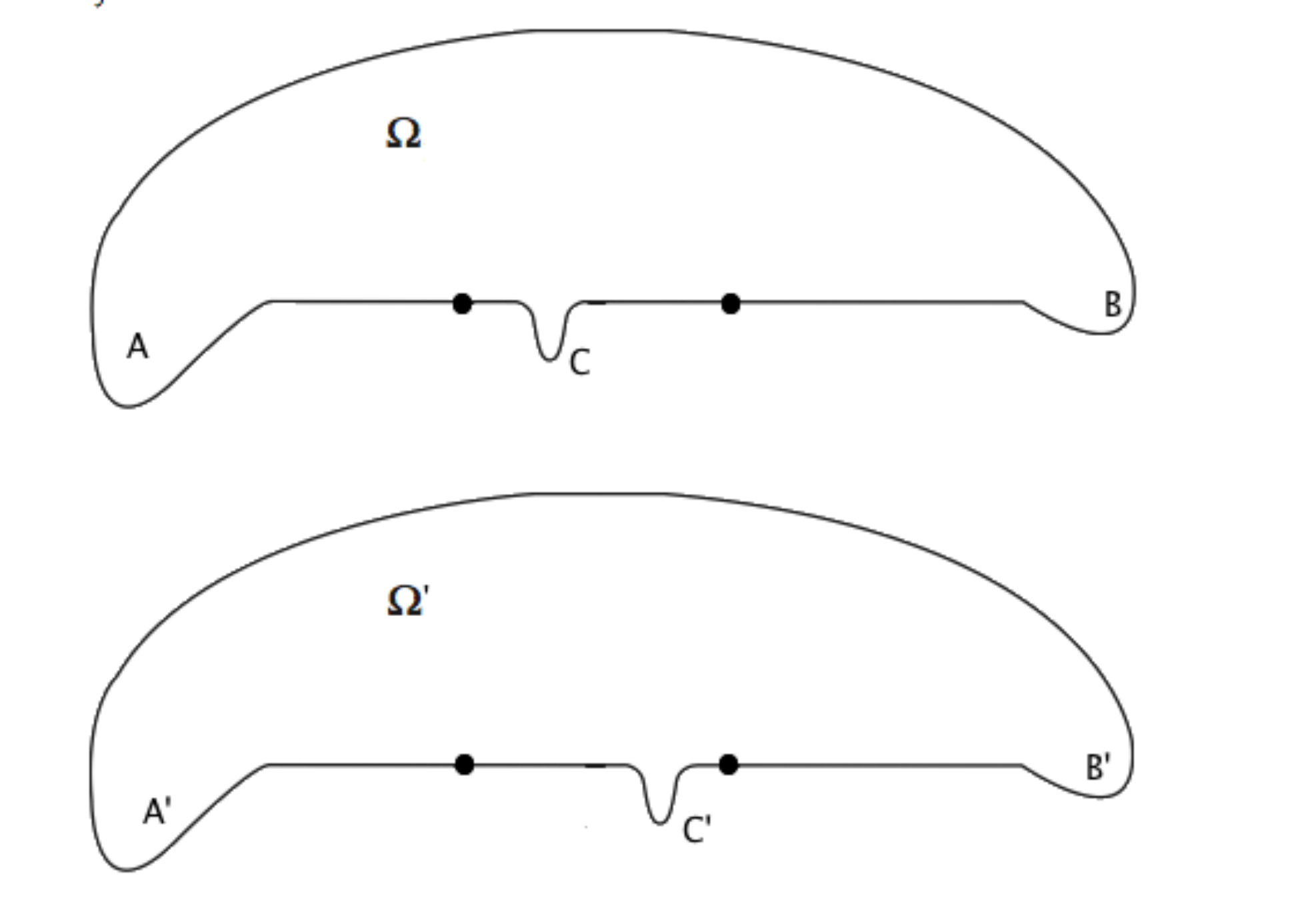}
\caption{ Two domains $\Omega$ and $\Omega'$ with $\Tr(\cos(t\sqrt{\Delta_\Omega})) - \Tr(\cos(t\sqrt{\Delta_{\Omega'}})) \in C^\infty(\R)$, but $\spec(\Delta_\Omega) \neq \spec (\Delta_{\Omega'})$. }
\label{mushrooms}
\end{figure}

An analogous result in the case of bounded plane domains was proved by Fulling and Kuchment in \cite{fulling} where they find two bounded plane domains (called Penrose-Lifshits mushrooms) for which all the wave trace invariants agree but the ground state eigenvalues disagree. A Penrose-Lifshits domain $\Omega$ is a semi-ellipse with asymmetrical bumps  $A$, $B$ and $C$ attached to its boundary as in Figure \ref{mushrooms}. If one detaches the bump $C$, reflects it about the axis of the ellipse and reattaches it, one gets a
non-congruent domain  $\Omega'$ which has the same heat and wave trace invariants as
$\Omega$, i.e. is indistinguishable from $\Omega$ by standard inverse spectral techniques.
In \cite{fulling}, Fulling and Kuchment prove Zelditch's conjecture (see \cite{zelditchsurvey}) that these domains are not isospectral.
  In \cite{c}, Colin de Verdi\`ere constructs an analogue of
the Penrose-Lifshits example for the $1D$ semiclassical Schr\"odinger
operator (see Figure \ref{PenrosePotentials}). Namely he attaches two small bump functions to  the harmonic
oscillator potential $V_0(x)=x^2$. By detaching one of these functions, reflecting
it about the $y$-axis and then reattaching it, he is able to construct two
non-isomorphic potentials that are isospectral modulo $h^\infty$. Our
goal in this paper is to prove a Fulling-Kuchment theorem for this example
and verify Colin de Verdi\`er's conjecture that these ``isospectral modulo $h^\infty$''
potentials are not isospectral.

 \vspace{.25in}

 \begin{figure}[htbp]
\includegraphics[width=12cm]{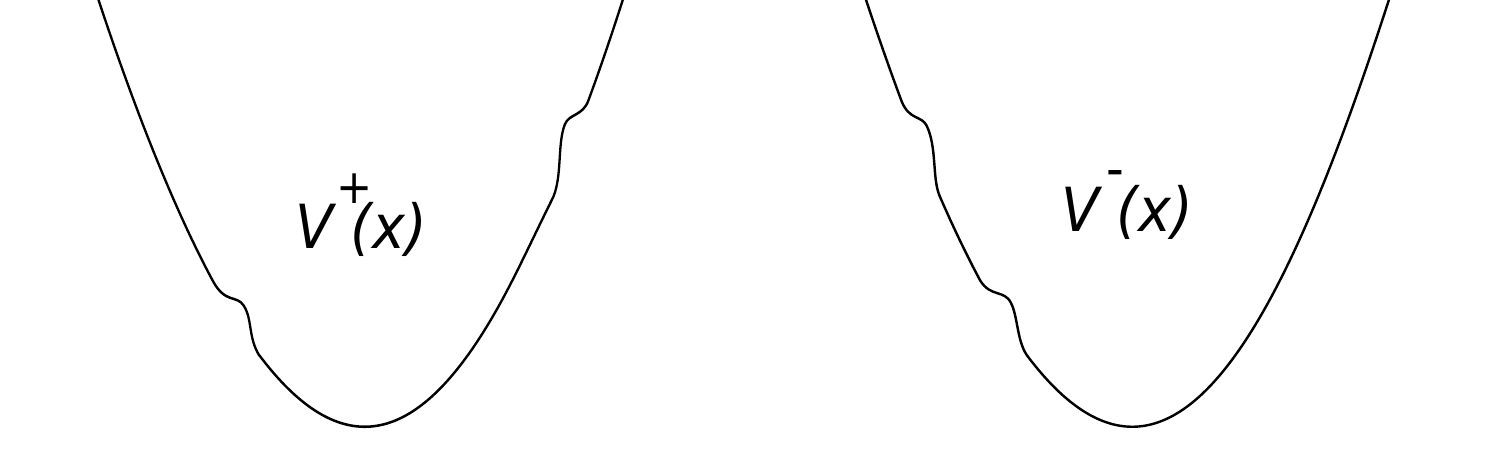}
\caption{ Two potentials with $\spec(P_{V^+}) =\spec(P_{V-})$ up to $\Oh(h^\infty)$ but  $\spec(P_{V^+}) \neq \spec(P_{V-})$} 
\label{PenrosePotentials}
\end{figure}

To our knowledge there are no counterexamples to spectral uniqueness of smooth
semiclassical Schr\"odinger operators. There are in fact many recent positive results
 in this area which use trace invariants but make
  strong assumptions of analyticity or symmetry on the potential. In \cite{gu07}, the first author and Uribe show that any real analytic potential $V$ on $\mathbb R^n$, symmetric with respect to all coordinate axes and with a unique global minimum, is determined within the class of all such potentials by the spectrum at the bottom of the well of its semiclassical Schr\"odinger operator. In \cite{h09}, the second author and in \cite{cg}, Colin de Verdi\`ere and the first author remove the symmetry assumption in dimension $1$, but keep the assumption of analyticity (see also \cite{h09,gu10} for the higher dimensional case). In \cite{c}, Colin de Verdi\`ere removes both the symmetry and analyticity assumptions in the one dimensional case, but adds a genericity assumption and uses all eigenvalues instead of only the low lying ones. In \cite{gw10}, the first author and Wang give a new proof of Colin de Verdi\`ere's result with slightly different generic conditions. In \cite{dhv}, Datchev, Ventura and the second author show that radially symmetric potentials in $\mathbb R^n$ are spectrally determined among all smooth potentials.

\subsection{Remarks} We close the introduction by listing some remarks and related problems:
\begin{itemize}
\item Our method cannot eliminate the possibility of the existence of a 
sequence $h_k \to 0$ where the ground state eigenvalues agree. 
We believe that such a sequence does not exist and we actually expect a much stronger statement to hold:
\[\text{For all} \; j\geq1,\; \text{there exist}\; c_j, C_j>0:\quad|\lambda_j^+-\lambda_j^-| \geq C_j e^{-c_j/h}.\]
\item We can also ask the same question for semiclassical resonances. Can we find two smooth compactly supported potentials $V^+$ and $V^-$ where $P_{V^+}$ and $P_{V^-}$ have different resonances but where the resonances agree up to $\Oh (h^\infty)$? 
\item It would be interesting to study the analogous problem in the case of compact Riemannian manifolds.
 To be more precise, can one find a pair of Riemannian manifolds $(M,g)$ and $(M',g')$
 where $\spec(\Delta_g)\neq \spec(\Delta_{g'})$ but all the wave trace invariants are the same, that is to say $\Tr e^{-it\sqrt{\Delta_g}} -\Tr e^{-it\sqrt{\Delta_{g'}}} \in C^\infty(\mathbb R)$?
\end{itemize}

We are grateful to Daniel Stroock for bringing to our attention a reference for the Kato-Rellich theorem. We are also thankful to Kiril Datchev for his useful comments on the earlier version of the paper. The second author would also like to Ramis Movassagh for helping us to find a graph of the Weber function using \textit{Maple}. 

\section{Proof of Theorem}
 \begin{proof} We will show the existence of the potentials $V^{\pm}$ in Theorem \ref{Theorem} by choosing suitable perturbations of the harmonic oscillator $V_0(x)=x^2$. Suppose $\alpha, \beta \in C^\infty_0(\mathbb R)$ such that supp$(\alpha) \subset (-3, -2)$ and supp$(\beta) \subset (3, 4)$ and they are not identically zero. We put
\begin{equation}\label{Vplusminus}
V^+(x)=x^2+t\alpha(x)+\epsilon\beta(x),
\end{equation}
$$V^-(x)=x^2+t\alpha(x)+\epsilon\beta(-x).$$ We denote by $\lambda^\pm_1(\epsilon)$ the ground state eigenvalue of $P_{V^\pm}$. We will show that there exist $\alpha, \beta$ and $t, \epsilon >0$ small enough such that the pair $V^\pm$ has the properties asserted in the theorem. From the beginning we assume that $\epsilon$ and $t$ are small enough that $V^\pm$ do not have any critical points beside $x=0$.

We first show that ``$\spec(P_{V^+})=\spec(P_{V-})$ up to $\Oh(h^\infty)$'' in the sense of \eqref{spectrums}. This is claimed and proved in \cite{c} using the method of ``Bohr-Sommerfeld quantization conditions to all orders''. It can also be proved using quantum Birkhoff normal forms  and their equivalence with trace invariants as in \cite{gw10}. In that paper the first author and Wang consider the spectral density measure $\nu_h$ defined by
\[ \nu_h(f)=\;\text{Tr}\;f(P_V)=\sum_{j=1}^\infty f(\lambda_j)\] and they show that it satisfies an asymptotic expansion of the form:
\[\nu_h(f)\sim (2\pi h)^{-n} \sum_{k=0}^\infty a_k(f) h^{2k},\]
where $a_k$ are called \textit{semiclassical trace invariants}. They show that these invariants are of the form:
\begin{equation}\label{traceinvariants}a_k(f)=\int \sum_{j=0}^k f^{(2j)}(\xi^2+V(x))P_{k,j}(V', \dots , V^{(2k)})dx d\xi, \end{equation}
where $P_{k,j}$ are universal polynomials. They also give an algorithm to compute $a_k$. Then they show that $\nu_h$ determines the global Birkhoff normal form (whose existence they also prove in this setting) of $P_V$ up to $\Oh(h^\infty)$ and thus determines the eigenvalues up to $\Oh(h^\infty)$. However for our potentials $V^+$ and $V^-$, the trace invariants $a_k$  in \eqref{traceinvariants} are identical and hence their $\nu_h$ are the same up to $\Oh(h^\infty)$ which proves the claim.

To show that the ground state eigenvalues $\lambda_1^+(\epsilon)$ and $\lambda_1^-(\epsilon)$ are different we use Hadamard's variational formula (Lemma \ref{Hadamard} below) which implies that
\begin{equation}\label{HDM}\left.\frac{d}{d\epsilon}\right|_{\epsilon=0} \lambda^\pm_1(\epsilon)=\int_{\mathbb R}\beta(\pm x) |u_1(x)|^2 dx, \end{equation}
 where $u_1(x)$ is an $L^2$ normalized eigenfunction of $-h^2 \frac{d^2}{dx^2} + x^2+t \alpha(x)$ with the 
ground state eigenvalue $\lambda_1=\lambda^+_1(0)= \lambda^-_1(0)$. Then in Lemma \ref{nonevenness} we show that for $h=1$ and $t$ small enough $|u_1(x)|^2$ is not an even function in $(-4, -3) \cup (3,4)$. Thus by (\ref{HDM}) we can find a $\beta$ such that:
$$\left.\frac{d}{d\epsilon}\right|_{\epsilon=0} \lambda^+_1(\epsilon) \neq \left. \frac{d}{d\epsilon}\right|_{\epsilon=0} \lambda^-_1(\epsilon).$$ This implies that for $\epsilon$ small enough (and $t$ small enough)
 we have:
$$\text{For} \; h=1: \quad \lambda^+_1(\epsilon) \neq \lambda^-_1(\epsilon).$$
Finally by the Kato-Rellich theorem (see \cite{RS} Theorem XII.8.) the eigenvalues $\lambda^\pm_1$  are analytic functions of $h$ for all $h\neq0$, therefore  $\lambda^+_1(\epsilon) \neq \lambda^-_1(\epsilon)$ for all $h>0$, except possibly for a sequence $h_k \to 0$. This finishes the proof of the theorem.\end{proof}
It now remains only to prove Lemmas \ref{Hadamard} and \ref{nonevenness} below. Before doing this we makes some remarks regarding the proof above. 

\begin{rem} We note that the Kato-Rellich theorem is very important in our proof. We could not follow our argument with a variable $h$ because $\epsilon$ would depend on $h$. We choose and fix an $\epsilon$ for $h=1$ and then use the analyticity in $h$ to argue that, except for a sequence $h_k \to 0$, for all $h>0$ we have $\lambda^+_1(\epsilon) \neq \lambda^-_1(\epsilon)$. 
\end{rem}
\begin{rem} \label{gs} We also point out that the eigenvalues $\lambda_1^\pm(\epsilon)$ are not analytic at $h=0$. We can see this using the theory of quantum Birkhoff normal forms at the bottom of the well of a potential which was developed by Sj\"ostrand in \cite{Sj}. Since $V^\pm(x)$ and $V_0(x)=x^2$ have the same Taylor coefficients at $x=0$, the bottom of of their wells, they have the same QBNFs at $(x,\xi)=(0,0)$ and therefore the low lying eigenvalues (in particular the ground states) of their Schr\"odinger operators $P_{V^\pm}$ and $P_{V_0}$ must have the same asymptotic expansion of the form $q_1 h+q_2 h^2 +\dots$ as $h \to 0$. However the ground state eigenvalue of $V_0=x^2$ is $h$ hence if $\lambda_1^\pm(\epsilon)$ was analytic at $h=0$ then we would have $\lambda_1^\pm(\epsilon)=h$. But this is not the case if for example we choose $\alpha$ and $\beta$ nonnegative and not identically zero. In fact under this assumption we have $\lambda_1^\pm >h$. To see this we recall that
\begin{equation}\label{groundstate}
\lambda^+_1=\min_{\phi: \, ||\phi||_{L^2}=1} \big((-h^2 \frac{d^2}{dx^2} + x^2+\rho(x) ) \phi, \phi \big) _{L^2}.
\end{equation} where $\rho=t\alpha+\epsilon \beta$ and by assumption $\rho \geq 0$. Let $u_1$ be an $L^2$ normalized ground state eigenfunction of $-h^2 \frac{d^2}{dx^2} + x^2+\rho(x)$. Then in (\ref{groundstate}) the minimum is attained by $u_1$ and
$$ \lambda^+_1=\big((-h^2 \frac{d^2}{dx^2} + x^2)u_1, u_1 \big)_{L^2} + \int \rho |u_1|^2 dx .$$ The first term is greater than or equal to $h$ with equality only if $u_1$ is a ground state eigenfunction of the harmonic oscillator in which case the second term is not zero. So $\lambda^+_1 >h$. 
\end{rem}

We now state Hadamard's variational formula for Schr\"odinger operators. It is a formula for the first variation
 of the eigenvalues of a perturbed operator in terms of the eigenfunctions of the unperturbed operator. Since in this case (unlike for example in the case of bounded domains) the argument is simple we give its proof.
\begin{lem}\label{Hadamard} Let $V \in C^\infty(\R;\R)$ with $\lim_{|x| \to \infty} V(x) = \infty$. Let $\beta \in C^\infty_0(\mathbb R)$. Suppose $\lambda_j(\epsilon)$ is the $j$-th eigenvalue of 
$-h^2 \frac{d^2}{dx^2} + V(x)+\epsilon \beta(x)$ and suppose $u_j(x)$ is an $L^2$ normalized eigenfunction of $-h^2 \frac{d^2}{dx^2} + V(x)$ with eigenvalue $\lambda_j(0)$. Then
$$\left.\frac{d}{d\epsilon}\right|_{\epsilon=0}\lambda_j(\epsilon)=\int_{\mathbb R}\beta(x) |u_j(x)|^2 dx.$$
\end{lem}
\begin{proof} Since in dimension one the eigenvalues are simple, for a given $j$ we can choose a smooth one parameter family $u_j(x,\epsilon)$ of $L^2$ normalized real eigenfunctions of $-h^2 \frac{d^2}{dx^2} + V(x)+\epsilon \beta(x)$ with eigenvalues $\lambda_j(\epsilon)$. So by our notation $u_j(x)=u_j(x,0)$. We now write
\[\begin{split}\left.\frac{d}{d\epsilon}\right|_{\epsilon=0}\lambda_j(\epsilon)&= \left.\frac{d}{d\epsilon}\right|_{\epsilon=0}
{\big((-h^2 \frac{d^2}{dx^2} + V(x)+\epsilon \beta(x))u_j(x,\epsilon),u_j(x,\epsilon)\big)_{L^2}}\\ &
=\big(\beta(x)u_j(x), u_j(x)\big)_{L^2}+\big((-h^2 \frac{d^2}{dx^2} + V(x))(\left.\frac{d}{d\epsilon}\right|_{\epsilon=0} u_j(x,\epsilon)),u_j(x)\big)_{L^2}\\ & \qquad \qquad +\big((-h^2 \frac{d^2}{dx^2} + V(x))u_j(x), \left.\frac{d}{d\epsilon}\right|_{\epsilon=0}u_j(x,\epsilon)\big)_{L^2}
\end{split}\]
Because $V$ is real valued the operator $-h^2 \frac{d^2}{dx^2} + V(x)$ is symmetric and therefore the last two terms are identical. In fact each one is zero. This follows from $ \big(u_j(x), \left.\frac{d}{d\epsilon}\right|_{\epsilon=0}u_j(x,\epsilon)\big)_{L^2}=0$ which in turn follows by applying $\left.\frac{d}{d\epsilon}\right|_{\epsilon=0}$ to the equation $ \big(u_j(x,\epsilon), u_j(x,\epsilon) \big)_{L^2}=1$. 
\end{proof}

In the next lemma we put $h=1$.
\begin{lem}\label{nonevenness} Let $\alpha \in C^\infty_0(-3,-2)$ be nonnegative and not identically zero and let $V(x)=x^2+t\alpha(x)$. Suppose $u_1(x)$ is a ground state eigenfunction of $- \frac{d^2}{dx^2} + V(x)$ with the ground state eigenvalue $\lambda_1$. Then there exists $t>0$ small enough such that $|u_1(x)|^2$ is not an even function on $(-4, -3) \cup (3,4)$.
\end{lem}
\begin{proof}
First of all since away from supp$(\alpha)$ the potential $V(x)$ is real analytic, if $|u_1(x)|^2$ is even on $(-4, -3)\cup(3,4)$ then by analytic continuation it is even on $(-\infty, -3)\cup(3, \infty)$. We also note that $(u_1(x))^2$ being even implies $u_1(x)$ is even. This is because a ground state eigenfunction never vanishes. Therefore $u_1(x)$ can not change sign and in particular can not be odd anywhere.

To prove $u_1(x)$ is not even on $(-\infty, -3)\cup(3, \infty)$ we introduce the parabolic cylinder functions below and review some of their basic properties in \S\ref{Weberfunction}.

Let $W(x)$ be the unique solution to 
\begin{equation}\label{Weber}-W''(x)+(x^2-\lambda_1)W(x)=0, \end{equation} with \[\lim_{x \to - \infty} W(x)=0 \; \quad \text{and}\; \quad W(-3)=u_1(-3).\]
The function $W(x)$ is called a Weber function. An argument using \eqref{groundstate}, similar to that in Remark \ref{gs}, shows that for $t$ small enough we have $1<\lambda_1<3$. In fact we can make $\lambda_1$ arbitrary close to $1$. By the WKB method we see that $|W(x)|$ either grows or decays exponentially as $x \to \infty$, but exponential decay is ruled out by the fact that it would make $W$ an eigenfunction of the harmonic oscillator with eigenvalue $\lambda_1 \not\in\{1,3,5,\dots\}$. Thus  $W(x)$ is exponentially decaying near $-\infty$ with $\lim_{x \to -\infty} W=0$ and exponentially growing (in absolute value) near $\infty$. In fact we will see in \S\ref{Weberfunction} that $\lim_{x \to \infty} W=-\infty$, and that $W$ has a unique critical point (which is a global maximum) at $x=-a$, $|a|<\sqrt{\lambda_1}$, and vanishes only once for a large positive value of $x$.  In particular, $W'(x)$ is positive for $x<-a$ and negative for $x>-a$. See Figure \ref{Weberfigure}  for a graph of $W$. We will prove these claims about $W$ in \S\ref{Weberfunction}.

Since  $u_1$ satisfies the same equation as $W$ on $(-\infty, -3)\cup (-2, \infty)$,  using $u(-3) = W(-3)$ we have
\[u_1(x)=W(x), \quad x \leq -3,\]
and using $\lim_{x \to \infty}u_1(x) = \lim_{x \to \infty} W(-x) = 0$ we have
\[u_1(x)=cW(-x), \quad x \geq -2\] 
for a positive constant $c$. In particular $u_1(0)=cW(0)$. We claim that $c>1$ which will show that $u_1$ is not even on $(-\infty, -3)\cup (3, \infty)$, completing the proof of the lemma.

The proof that $c>1$ uses Sturm-Liouville theory. Let $Q(x)=\lambda_1-x^2$ and $Q_1(x)=\lambda_1-x^2-t\alpha(x)$. Then
\begin{equation}\label{Q} W''+QW=0,
\end{equation}
and
\begin{equation}\label{Q_1}u_1''+Q_1u_1=0.
\end{equation}
We rewrite \eqref{Q} and \eqref{Q_1} as a first order system using the Pr\"ufer substitution:
\begin{equation}\label{polar}
\left\{\begin{array}{ll}
W(x)=r(x)\sin \theta(x) \\
W'(x)=r(x)\cos \theta(x)
\end{array} \right., \qquad 
\left\{\begin{array}{ll}
u_1(x)=r_1(x)\sin \theta_1(x) \\
u_1'(x)=r_1(x)\cos \theta_1(x)
\end{array} \right. ,
\end{equation}
where we choose the branches of $\theta$ and $\theta_1$ such that $\theta(-3)=\theta_1(-3) \in (0,\pi/2)$ (this is possible because $W(-3)>0$ and $W'(-3)>0$, see \S\ref{Weberfunction}).
Then (see \cite{BR} section 5, chapter 10):
\begin{equation}\label{dtheta}
\left\{\begin{array}{ll}
\theta'(x)=Q(x) \sin^2 \theta(x) + \cos^2 \theta(x),
\\ \theta_1'(x)=Q_1(x) \sin^2 \theta_1(x) + \cos^2 \theta_1(x), \end{array} \right.
\end{equation}
and by \eqref{polar}
\begin{equation}\label{cot}
\left\{\begin{array}{ll}
W'(x)=\cot \theta(x) W(x),
\\ u_1'(x)=\cot \theta_1(x) u_1(x).\end{array} \right.
\end{equation}
Because $Q_1(x)\leq Q(x)$ and because $\theta_1(-3)=\theta(-3)$, by applying an elementary comparison theorem  to \eqref{dtheta} (see \cite[Theorem 7, Chapter 1]{BR}) we get
\begin{equation}\label{comparison1}
\theta_1(x) \leq \theta(x) \quad \text{for} \quad x\geq -3.
\end{equation}
Moreover since $W(x)$ and $W'(x)$ are positive on $x<-a$ (see \S\ref{Weberfunction}) we have $0<\theta(x)<\frac{\pi}{2}$ for $x<-a$ and thus
$\theta_1(x)< \frac{\pi}{2}$ for $-3 \leq x<-a$. In fact, because $u_1$ never vanishes, $\theta_1(x)$ is never zero and
\[ 0<\theta_1(x) \leq \frac{\pi}{2} \quad \text{for} \quad -3 \leq x \leq -a.\]
Using this and \eqref{comparison1}, by applying an elementary comparison theorem to \eqref{cot} and because $u_1(-3)=W(-3)$ we obtain 
\begin{equation} \label{comparison2}
u_1(x) \geq W(x) \quad \text{for} \quad -3\leq x \leq -a. 
\end{equation} 
This is the key point: \eqref{comparison2} is an inequality for $u_1$ and $W$ in the region where they solve different equations.
On the other hand $u_1(x)=cW(-x)$ for $x \geq -2$ so
\begin{equation} \label{u1'}u_1'(x)=-c W'(-x) \quad \text{for} \quad x \geq -2, \end{equation} and in particular
$u_1'(-a)=-c W'(a)$.  If $\theta_1(-a)=\frac{\pi}{2}$ then $u_1'(-a)=0$ by \eqref{polar} and therefore $W'(a)=0$. But $W$ has only one critical point which is at $x=-a$.  Hence $\theta_1(-a) < \frac{\pi}{2}$. Then by \eqref{polar} $u'(-a) >0$ and thus $W'(a)<0$. This shows that $a$ must be positive because if $a<0$ then the point $a$ is on the left of the critical point $-a$ of $W$ but $W$ is increasing for $x<-a$. Finally $W$ is decreasing on $(-a,0)$ and by \eqref{u1'} $u_1$ is increasing on $(-a,0)$. So by \eqref{comparison2} $u_1(0) >W(0)$ and therefore $c>1$.
\end{proof}

\section{Parabolic cylinder functions (Weber functions) with small frequency}\label{Weberfunction} Let $W(x)$ be a Weber function with small frequency which was defined in \eqref{Weber}. Here we prove the properties of $W$ which are needed in the proof of Lemma \ref{nonevenness}.

That $W(x)$ decays exponentially near $-\infty$ and grows exponentially near $\infty$ follows from the classical WKB method for ODEs (see for example \cite{Olver}). In fact the exact decay and growth rates (see \cite{WW}, section 16.5.) are given by
\[ W(x) \sim C (\sqrt{2}|x|)^{(\lambda_1-1)/2} e^{-x^2/2}  \quad \text{as} \quad x \to -\infty, \]
\[ W(x) \sim \frac{C\sqrt{2\pi}}{\Gamma(\frac{1}{2}(1-\lambda_1))}(\sqrt{2}x)^{-(\lambda_1+1)/2} e^{x^2/2}
 \quad \text{as} \quad x \to \infty,\] where $C$ is a positive constant which depends on our normalization $W(-3)=u_1(-3)$. Note that for $\lambda_1>1$ very close to $1$ we have $ \Gamma(\frac{1}{2}(1-\lambda_1))<0$ which shows that $\lim_{x \to \infty}W(x)=-\infty$. 
 
 \begin{figure}[htbp]
\includegraphics[width=12cm]{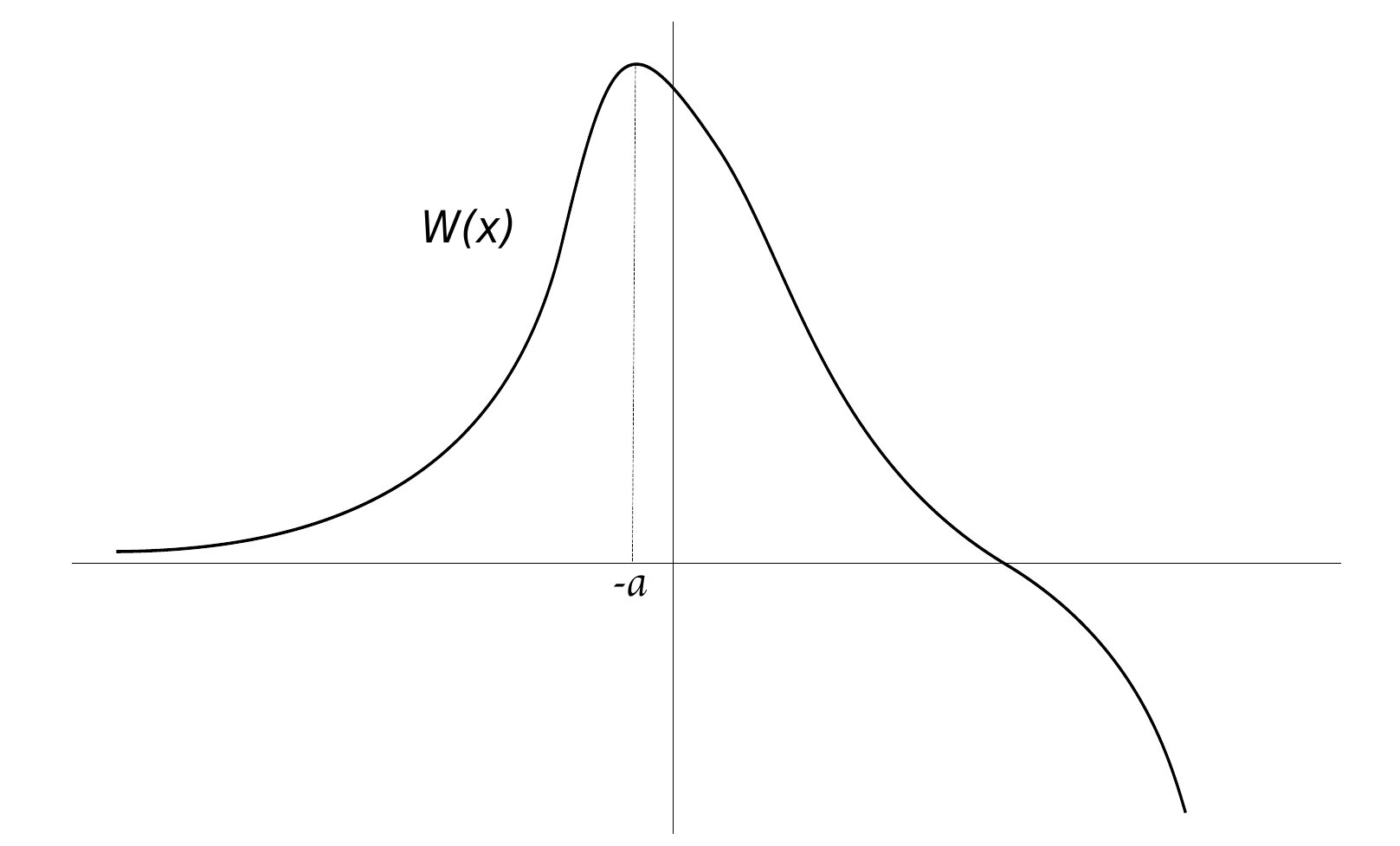}
\caption{ Weber function with small frequency.  }\label{Weberfigure}
\end{figure}
 
To prove the other properties of $W$  we first show that $W$ is positive on $(-\infty, 3]$.  To do this we choose $\eta \in C_0^\infty(3,4)$ nonnegative and consider the operator  
\[
-\frac{d^2}{dx^2}+x^2+\delta \eta(x).
\]
We represent its ground state eigenvalue by $\mu_1(\delta)$ and an associated smooth family of $L^2$ normalized eigenfunctions by $\psi_1(x,\delta)$. Then by Hadamard's variational formula
\[ \frac d {d \delta} \mu_1(\delta) =\int \eta(x) |\psi_1(x,\delta)|^2 dx.\]  Similarly
\[ \frac d {d t} \lambda_1(t) =\int \alpha(x) |u_1(x,t)|^2 dx.\] This implies that $\mu_1(\delta)$ and $\lambda_1(t)$ are increasing functions in $\delta$ and $t$ respectively. Thus because $\mu_1(0)=\lambda_1(0)=1$, we can find $\delta$ and $t$ positive so that $1<\mu_1(\delta)=\lambda_1(t)<2$. This is the value $t$ that we choose in Lemma \ref{nonevenness}. Because $\psi_1$ satisfies the same equation as $W$ for $x\leq 3$ we have $\psi_1=kW$ on $(-\infty,3]$. On the other hand since $\psi_1$ is  a ground state eigenfunction it does not vanish on $\mathbb R $ so the Weber function $W$ has no zeros in $(-\infty,3]$ and must be positive there.

We can use this fact to show that there is only one critical point which is at the maximum $x=-a$. First of all in the forbidden region $x <-\sqrt{\lambda}$  we have $W(x)>0$ and therefore by \eqref{Weber} we have $W''>0$ there. This implies that $W'$ will increase on $(-\infty, -\sqrt{\lambda_1})$ and in particular will never vanish.  The maximum $x=-a$ must be in the classical region $|x|\leq \sqrt{\lambda_1}$ where $W''$ and $W$ have different signs. There can not be two critical points in the classical region because then by Rolle's theorem $W''$ would vanish in between them, and $W$ would vanish there as well by the equation, contradicting $W>0$.  Furthermore, $W$ can not have any critical points in the forbidden region $x > \sqrt{\lambda_1}$. Assume $x_0$ is such a critical point. Then $W''(x_0)\neq0$ because if it were zero then by \eqref{Weber} $W(x_0)=0$ which together with would $W'(x_0)=0$ imply that $W$ is a trivial solution. If $W''(x_0)>0$ then by \eqref{Weber} $W(x_0)>0$ and the graph of $W$ should stay concave up for all $x>x_0$ which implies that $\lim_{x\to \infty }W(x)=\infty$ which is a contradiction. Finally if $W''(x_0)<0$ then $W(x_0) <0$ but because $W(3)>0$ we should have a zero of $W$ between $3$ and $x_0$ which shows that $x_0>3$. Now if $z_0$ is the smallest zero with $z_0>3$ then $W'(z_0)<0$. After this point $z_0$, $W'$ can only decrease (and is never zero) because $W''<0$ in this region. This contradicts $W'(x_0)=0$.

This concludes the proof of the properties of $W$ which are needed for Lemma \ref{nonevenness}. In the course of the proof of this lemma we also prove that $a>0$.

\end{document}